\documentclass[oneside,12pt]{amsart}
\usepackage{amssymb, amscd, amsthm}
\usepackage{url}

\usepackage[utf8]{inputenc}
\usepackage[T1]{fontenc}
\usepackage{mathtools}
\usepackage{tikz,tikz-cd}
\usetikzlibrary{arrows, matrix}

\usepackage{enumitem}
\setlist[enumerate,1]{label=\arabic*., ref=\arabic*, leftmargin=*}

\usepackage{multicol}

\usepackage{geometry}

\newtheorem{thm}{Theorem}[section]
\newtheorem{cor}[thm]{Corollary}
\newtheorem{lem}[thm]{Lemma}
\newtheorem{df}[thm]{Definition}
\newtheorem{prop}[thm]{Proposition}

\newtheorem{rem}[thm]{Remark}

\newcommand{\N}{\ensuremath{\mathbb N}}
\newcommand{\Z}{\ensuremath{\mathbb Z}}

\newcommand{\R}{\ensuremath{\mathbb R}}

\newcommand{\G}{\Gamma}
\newcommand{\La}{\Lambda}

\DeclareMathOperator{\vspan}{span}
\DeclareMathOperator{\Spin}{Spin}
\DeclareMathOperator{\SO}{SO}
\DeclareMathOperator{\GO}{O}
\DeclareMathOperator{\GL}{GL}
\DeclareMathOperator{\SL}{SL}

\DeclareMathOperator{\diag}{diag}

\newcommand{\st}{\ensuremath{\;|\;}}

\newcommand{\ba}{\bar{\alpha}}
\newcommand{\bb}{\bar{\beta}}
\newcommand{\bg}{\bar{\gamma}}

\newcommand{\tf}{\widetilde{F}}

\newcommand{\seslr}[7]{#1 \longrightarrow #2 \stackrel{#3}{\longrightarrow} #4 \stackrel{#5}{\longrightarrow} #6 \longrightarrow #7}

\newcommand{\sesr}[6]{\seslr{#1}{#2}{}{#3}{#4}{#5}{#6}}

\title[]{Classification of spin structures on 4-dimensional almost-flat manifolds}

\author{R.~Lutowski}
\address{Institute of Mathematics, University of Gda\'{n}sk, Gda\'{n}sk, Poland}%
\email{rafal.lutowski@mat.ug.edu.pl}%

\author{N.~Petrosyan}
\address{Department of Mathematics, University of Southampton, Southampton, UK}%
\email{N.Petrosyan@soton.ac.uk}%

\author{A.~Szczepa\'{n}ski}
\address{Institute of Mathematics, University of Gda\'{n}sk, Gda\'{n}sk, Poland}%
\email{aszczepa@mat.ug.edu.pl}%

\thanks{\noindent The first and third authors were supported by the Polish National Science Center grant 2013/09/B/ST1/04125. The second author was supported by the EPSRC First Grant 	EP/N033787/1.}

\newsavebox\CBox
\def\textBF#1{\sbox\CBox{#1}\resizebox{\wd\CBox}{\ht\CBox}{\bfseries\boldmath{#1}}}

\newcommand{\eemrow}{\global\let\currentrowstyle\relax}
\eemrow
\usepackage{array}
\usepackage{collcell}
\newcolumntype{E}{>{\collectcell\currentrowstyle}r<{\endcollectcell}}
\newcolumntype{L}{>{\collectcell\currentrowstyle}l<{\endcollectcell}}
\newcolumntype{R}{>{\collectcell\currentrowstyle}r<{\endcollectcell}}
\newcommand{\rowstyle}[1]{\gdef\currentrowstyle{#1}}

\newcommand{\emrow}{\rowstyle{\textBF}}

\begin{document}
\maketitle

\begin{abstract}  Almost-flat manifolds were defined by Gromov as a natural generalisation of flat manifolds and as such share many of their properties. Similarly to flat manifolds, it turns out that the existence of a spin structure on an almost-flat manifold is determined by the canonical orthogonal representation of its fundamental group. Utilising this, we classify the spin structures on all four-dimensional almost-flat manifolds that are not flat.  Out of {127} orientable families, we show that there are exactly {15} that are non-spin, the rest are in fact parallelizable.
\end{abstract}

\section{Introduction}
For a connected and simply-connected nilpotent Lie group $N$ and a compact subgroup $C$ of ${\rm Aut}(N)$,  the semi-direct product $N\rtimes C$ acts on $N$ by 
$$(n,\varphi)\cdot m= n\varphi(m) \;\;\; \forall m, n\in N \mbox{ and } \forall \varphi \in C.$$
An {\it almost-crystallographic group} is a discrete subgroup  $\Gamma\subseteq N\rtimes C$ that acts  co-compactly on $N$.  If $\G$ is torsion-free, then it is said to be {\it almost-Bieberbach}. In this case, the quotient space $N/\G$ is a closed manifold called an {\it infra-nil} manifold (modelled on $N$). 

An {\it almost-flat} manifold is a closed manifold $M$ such that for any $\varepsilon >0$ there exists a Riemannian metric $g_{\varepsilon}$ on $M$ with $|K_{\varepsilon}|\mathrm{diam}(M, g_{\varepsilon})^2<\varepsilon$ where $K_{\varepsilon}$ is the sectional curvature and  $\mathrm{diam}(M, g_{\varepsilon})$ is the diameter of $(M, g_{\varepsilon})$. 

By the results of  Gromov \cite{Gr78} and Ruh \cite{Ru82}, the classes of almost-flat manifolds and infra-nil manifolds coincide. 
Such manifolds occur naturally in the study of  Riemannian manifolds with negative sectional curvature. For example, every complete non-compact finite volume manifold with pinched negative sectional curvature has finitely many boundary components which are all almost-flat manifolds (see \cite[\S1]{BK81}). Another important occurrence  of such manifolds is in the study of collapsing manifolds with uniformly bounded sectional curvature. By a theorem of Cheeger-Fukaya-Gromov,  if  a manifold  is  sufficiently collapsed relative to the size of its diameter, then it admits a local fibration structure whose fibers are almost-flat (see \cite{CFG92}).

It is well-known that all closed orientable manifolds of dimension at most three have a spin structure (see \cite[p.~35]{Ki89}, \cite[Exercise 12.B and VII, Theorem 2]{MS74}).  
In \cite{GPS16}, a complete list of non-spin $4$-dimensional orientable almost-flat manifolds with  cyclic $2$-group holonomy was given. In the current paper, we introduce an algorithm to determine the existence of a spin structure on an almost-flat manifold using the classifying representation of its fundamental group. We apply it to classify spin structures on all almost-flat manifolds in dimension four that are not flat. The case of $4$-dimensional flat manifolds has been done in \cite{PS10}. In doing so, we follow the classification of $4$-dimensional infra-nil manifolds given in \cite{D96}. 
There are {127} orientable families of four-dimensional infra-nil manifolds. Our computations show that there are exactly {15} families that cannot have a spin structure. In Corollary \ref{trivial}, we show that the remaining families of $4$-manifolds are in fact parallelizable, i.e.~have trivial tangent bundle.

We should also point out that we find one family (no.~80) of non-spin manifolds with holonomy group {$C_4$} which was inadvertently omitted in \cite{GPS16}.

\section{Preliminaries}

We will always assume that an almost-flat manifold comes equipped with a structure of an infra-nil manifold when discussing its topological properties. We denote by $\mathrm{O}(n)$ the real  orthogonal group of rank $n$. It is well-known that infra-nil manifolds are classified by their fundamental group which is almost-crystallographic.  A classical result of Auslander (see \cite{Au60}) asserts that every almost-crystallographic subgroup $\G \subseteq \mathrm{Aff}(N):=N\rtimes \mathrm{Aut}(N)$ fits into an extension
\[
\sesr{1}{\La}{\G}{\pi}{F}{1}
\]
where $\Lambda=\G\cap N$ is a uniform lattice in $N$ and $F$ is a finite subgroup of $C\subseteq \mathrm{Aut}(N)$  called the {\it holonomy group} of the corresponding infra-nil manifold $N/\G$. 

We recall a key result from \cite{GPS16} which states that the classifying map of the tangent bundle of an almost-flat manifold $M$ factors through the classifying space of the holonomy group $F$.  

\begin{prop}[{\cite[Prop.~2.1]{GPS16}}] \label{reduce}
Let $M$ be an orientable $n$-dimensional almost-flat manifold modelled on a connected and simply connected nilpotent Lie group $N$. Denote by $\G$ the fundamental group of $M$ and let
\[
\sesr{1}{\La}{\G}{\pi}{F}{1}
\]  
be the standard extension of $\G$. Then the classifying map $\tau:M\to {BSO}(n)$ of the tangent bundle of $M$ factors through $BF$ and is induced by a composite homomorphism $\rho:\G{\buildrel \pi\over \to} F {\buildrel \nu\over \to} \SO(n)$ where $\nu: F \hookrightarrow {\SO}(n)$ is the holonomy representation.
\end{prop}

We denote by $$\Lambda=\gamma_1(\Lambda)> \gamma_2(\Lambda)>\cdots >\gamma_{c+1}(\Lambda)=1,$$ 
the  lower central series of $\Lambda$, i.e. $\gamma_{i+1}(\Lambda)=[\Lambda, \gamma_i(\Lambda)]$ for $1\leq i\leq c$. By Lemma 1.2.6 of \cite{D96}, we have that the isolator $\sqrt[\Lambda]{\gamma_i(\Lambda)}=\Lambda \cap {\gamma_i(N)}$, where $\gamma_i(N)$, $1\leq i\leq c+1$, form the lower central series of $N$. By Lemmas 1.1.2-3 of \cite{D96}, the resulting {\it adapted lower central series}
\begin{equation}\label{eq}
\Lambda=\sqrt[\Lambda]{\gamma_1(\Lambda)}>  \sqrt[\Lambda]{\gamma_2(\Lambda)}>\cdots >\sqrt[\Lambda]{\gamma_{c+1}(\Lambda)}=1,
\end{equation}
has torsion-free factor groups 
$$H_i={\sqrt[\Lambda]{\gamma_i(\Lambda)}\over{\sqrt[\Lambda]{\gamma_{i+1}(\Lambda)}}},\;\;\; 1\leq i\leq c.$$ Thus, each $H_i\cong \Z^{m_i}$  for some positive integer $m_i$.
Conjugation in $\Gamma$ induces an action of the holonomy group $F$ on each factor group $H_i$. This gives  a faithful representation 
$$\theta:F\hookrightarrow \mbox{GL}(m_1, \Z)\times \cdots \times  \mbox{GL}(m_c, \Z)\hookrightarrow \mathrm{GL}(n, \Z), \;\;\; m_1+\cdots + m_c=n,$$ 
which we call the {\it integral holonomy representation}.
\begin{prop}[{\cite[Prop.~2.4]{GPS16}}]\label{equiv} The integral holonomy representation $\theta:F\hookrightarrow \mathrm{GL}(n, \Z)$ is $\R$-equivalent to $\nu:F\hookrightarrow \mathrm{O}(n)\subseteq \mathrm{GL}(n)$.
\end{prop}

We may conclude that the almost-flat manifold $M$ is orientable if and only if the image of the representation $\theta:F\hookrightarrow   \mathrm{GL}(n, \Z)$ lies inside  $\mathrm{SL}(n, \Z)$.

\section{Spin representations }

The goal of this section is to present tools necessary to determine spin structures on almost-flat manifolds and discuss the main algorithm for the calculations. We start with the basic definitions which allows us to introduce some notation. We refer the reader to the original sources \cite{ABS64}, \cite{Ch54} and the textbook \cite{F00} for a comprehensive treatment of spin structures and its applications to the Dirac operator.

\begin{df}\rm
Let $n \in \N$. The \emph{Clifford algebra} $C_n$ is a~real unital associative algebra generated by the elements $e_1,\ldots,e_n$ satisfying the relations:
\[
\forall_{1 \leq i < j \leq n} \; e_i^2 = -1 \text{ and } e_ie_j = - e_je_i.
\]
\end{df}

\begin{rem}\rm
We may view $\R^n = \vspan\{e_1, \ldots, e_n\}$ as a~vector subspace of $C_n$, for $n \in \N$.
\end{rem}

Let $n \in \N$. In the Clifford algebra $C_n$ we have the following three involutions:
\begin{itemize}
\item[(i)] $^* \colon C_n \to C_n$, defined on the basis of (the vector space) $C_n$ by
\[ 
\forall_{1 \leq i_1 < i_2 < \ldots < i_k \leq n} \; ( e_{i_1}\ldots e_{i_k})^* = e_{i_k} \ldots e_{i_1};
\]
\item[(ii)] $' \colon C_n \to C_n$, defined on the generators of (the algebra) $C_n$ by 
\[
\forall_{1 \leq i \leq n} e_i' = -e_i.
\]
\item[(iii)] $\overline{ \phantom{a} } \colon C_n \to C_n$ is the composition of the previous involutions
\[
\forall_{a \in C_n} \; \overline{a} = (a')^*.
\]
\end{itemize}
The {\it spin group} is the following subgroup of the group of units in the Clifford algebra:
\[
\Spin(n) := \{ x \in C_n \st x' = x, x\overline{x}=1 \}.
\]
Moreover, there exists a covering map $\lambda_n \colon \Spin(n) \to \SO(n): x\mapsto  \lambda_n(x)$ defined by 
\begin{equation}
\label{eq:lambda}
 \forall_{v \in \R^n} \; \lambda_n(x)v = xv\overline{x},
\end{equation}
which is a homomorphism with kernel equal to $\{ \pm 1 \}$. For $n \geq 3$, the group $\Spin(n)$ is simply-connected (see for example \cite[Prop. 6.1, page 86]{Sz12},\cite[page 16]{F00}).

A {\it spin structure} on a smooth orientable manifold $M$ is an equivariant lift of its orthonormal frame bundle via the covering $\lambda_n$.  The existence of such a  lift is equivalent to the existence of a lift $\tilde{\tau}:M\to B\mathrm{Spin}(n)$ of the classifying map of the tangent bundle $\tau:M\to B\mathrm{SO}(n)$ 
such that $B_{\lambda_n}\circ\tilde{\tau}={\tau}$.  Equivalently, $M$ has a spin structure if and only if  the second Stiefel-Whitney class $w_2(T M)$ vanishes \cite[page 350]{BH59}.

Now, let $M=N/\G$ be an $n$-dimensional orientable almost-flat manifold modelled on the nilpotent group $N$ with the fundamental group $\G$, that fits into the standard extension
\begin{equation}
\label{eq:abgroup}
\sesr{1}{\La}{\G}{\pi}{F}{1}.
\end{equation}
Let $\rho \colon \G \to \SO(n)$ be the classifying representation of $\G$.  The following corollary of  Proposition \ref{reduce} allows us to classify the spin structures on $M$ via genuine group homomorphisms from its fundamental group to $\mathrm{Spin}(n)$ that lift the classifying representation. 

\begin{cor}[{\cite[Cor.~2.3]{GPS16}}] Let $M$ be an orientable almost-flat manifold of dimension $n$ with  
fundamental group $\Gamma$.
Then the set of spin structures on M is in one-to-one correspondence with the set of homomorphisms 
$\epsilon : \Gamma \to \mathrm{Spin}(n)$ such that $\lambda_n\circ\epsilon
= \rho$, i.e. the diagram
\begin{equation}
\label{eq:spin}
\begin{tikzcd}[ampersand replacement=\&]
 \& \Spin(n)\arrow{d}{\lambda_n} \\
\G \arrow{r}{\rho} \arrow[dashed]{ur}{\epsilon} \& \SO(n)
\end{tikzcd}
\end{equation}
commutes.
\end{cor}
\begin{rem}\rm In \cite{P00}, Pf\"{a}ffle used the analogous result (Proposition 3.2) to classify the Dirac spectra for all three dimensional orientable flat manifolds.
\end{rem}

\subsection{The algorithm for computing spin structures}  For an almost-Bieberbach group $\G$, there always exists a finite presentation, say $\G = \langle S \;|\; R \rangle$ with the set of generators $S$ and the set of relations $R$. A map $\epsilon'\colon S \to \Spin(n)$ can be extended to a homomorphism $\epsilon \colon \G \to \Spin(n)$ if and only if it preserves the relations of $\G$:
\[
\forall_{r_1,\ldots,r_l \in S} \forall_{\alpha_1,\ldots,\alpha_l \in \Z} \;\; r_1^{\alpha_1} \cdot \ldots \cdot r_l^{\alpha_l} \in R \Rightarrow \epsilon'(r_1)^{\alpha_1} \cdot \ldots \cdot \epsilon'(r_l)^{\alpha_l} = 1.
\]
Moreover, since $\ker \lambda_n = \{\pm 1\}$, in order to obtain commutativity of the diagram \eqref{eq:spin}, we must have
\[
\forall_{\gamma \in \G} \forall_{x \in \Spin(n)}\;\;  \rho(\gamma) = \lambda_n(x) \Rightarrow \epsilon(\gamma) \in \{\pm x\}.
\]

\noindent Hence, it follows that we can construct a lifting homomorphism $\epsilon$ if and only if for every generator $s \in S$, there exists an element $x_s \in \Spin(n)$ such that $\rho(s) = \lambda_n(x_s)$ and there is a combinations of signed elements $x_s \in \Spin(n)$, $s \in S$, that preserve the given relations of $\G$.

Note that in order to find a preimage in $\lambda_n^{-1}(g)$ of an element $g \in \SO(n)$ it is enough to use the definition of the group $\Spin(n)$ and the formula \eqref{eq:lambda} which gives us a system of $n$ linear equations with $2^n$ variables.
In general, this approach is not very efficient. The following lemmas allow us to simplify this problem in our context.

\begin{lem}[{\cite[Lemma 2.6]{GPS16}}]
\label{sylow} 
Let $M$ be an orientable almost-flat manifold  with the fundamental group $\G$ satisfying \eqref{eq:abgroup}.  Let  $S$ be  a 2-Sylow subgroup of $F$.  Then $M$ has a spin structure if and only if $N/{\pi^{-1}(S)}$ has a spin structure.
\end{lem}

\begin{lem}[{\cite[page 282]{LP15}}]
Let $\theta \colon F \to \GL(n,\Z)$ be a representation of a 2-group $F$. Then $\nu$ is $\R$-equivalent to a representation $\theta' \colon F \to \GO(n,\Z) = O(n) \cap \GL(n,\Z)$.
\end{lem}

To determine whether $M$ has a spin structure, using the above two lemmas and Proposition \ref{equiv}, without loss of generality, we can assume that $F$ is a~2-group and that the holonomy representation is of the form
\[
\theta \colon F \to \SO(n,\Z) = \SO(n) \cap \SL(n,\Z).
\]
The group $\SO(n,\Z)$ is generated by the matrices of the form
\[
P'_{(p\;q)} = \diag(\underbrace{1,\ldots,1}_{p-1},-1,1,\ldots,1) \cdot P_{(p\;q)}
\]
where $p < q$ and $P_{(p\;q)}$ is a matrix of the inversion $(p\;q)$.
An easy calculation shows that
\[
\lambda_n\left( \pm\frac{1+e_p e_q}{\sqrt{2}} \right) = P'_{(p\;q)}
\]
for $p<q$ and as a consequence
\[
\lambda_n(\pm e_{n_1} \cdots e_{n_l}) = D'
\]
where $D' \in \SO(n,\Z)$ is the diagonal matrix with $-1$ in exactly the entries $n_1,\ldots,n_l$ (see \cite[Lemma 7]{LP15}).

Lastly, we can simplify our calculation of spin structures on $M$ as follows. Let $\La^2$ be the normal closure of the group generated by the squares of the generators of the nilpotent lattice $\La$. If the lifting homomorphism $\epsilon \colon \G \to \Spin(n)$ exists, then $\La^2 \subseteq \ker\epsilon$. Therefore, in the following commutative diagram,  the existence of the lifting $\epsilon$ is equivalent to the existence of the lifting $\epsilon_2$:
\[
\begin{tikzcd}[ampersand replacement=\&]
\G/\La^2 \arrow[dashrightarrow]{r}{\epsilon_2} \& \Spin(n)\arrow{d}{\lambda_n} \\
\G \arrow{u}{\mu} \arrow{r}{\rho} \arrow[dashrightarrow]{ur}{\epsilon} \& \SO(n)
\end{tikzcd}
\]
where $\mu\colon \G \to \G/\La^2$ denotes the natural quotient homomorphism. We have proved:
\begin{prop}\label{mod2}
Let $M$ be an orientable almost-flat manifold of dimension $n$ with fundamental group $\Gamma$ satisfying  \eqref{eq:abgroup}.
Then the set of spin structures on M is in one-to-one correspondence with the set of homomorphisms 
$\epsilon_2 : \Gamma/\La^2 \to \Spin(n)$ such that $\lambda_n\circ\epsilon_2\circ\mu
= \rho$.
\end{prop}
\noindent Reducing  \eqref{eq:abgroup} mod $\La^2$, we have
\begin{equation}
\label{eq:abgroupmod2}
\sesr{1}{\La/\La^2}{\G/\La^2}{\pi_2}{F}{1}
\end{equation}
and we deduce:
\begin{cor}
Let $M$ be as above and $\rho_2=\nu\circ \pi_2: \Gamma/\La^2 \to \SO(n)$. Then the set of spin structures on M is in one-to-one correspondence with the set of homomorphisms 
$\epsilon_2 : \Gamma/\La^2 \to \Spin(n)$ such that $\lambda_n\circ\epsilon_2 = \rho_2$.
\end{cor}

\noindent Since the group  $\G/\La^2$ is finite, the above criteria is particularly useful when working with a computer.

\section{Preimages in Spin(4)}

Let $\G$ be a 4-dimensional almost-Bieberbach group with the standard extension \eqref{eq:abgroup}. In this section, we will calculate the subgroup $\lambda^{-1}(F)$ in $\Spin(4)$ where we denote $\lambda=\lambda_4$ to shorten the notation. 
For this, we will use the classification of almost-Bieberbach groups of dimension four with 2  and 3-step nilpotent Fitting subgroup given in Sections 7.2 and 7.3 of \cite{DE02}, respectively (see also \cite[page 197]{DE02} for an addendum to this list). Note that the case of four-dimensional Bieberbach groups, i.e.~almost-Bieberbach groups with abelian Fitting subgroup, was considered in \cite{PS10}.
We combine and enumerate them in families according to the isomorphism type of the holonomy group $F$.
In the classification given in \cite{D96}, there are always at most three generators of $\G$ which lie outside $\La$. We denoted them by $\alpha, \beta$ and $\gamma$.  
We will denote the induced generators of $F$ by
\[
\ba = \pi(\alpha), \bb = \pi(\beta) \text{ and } \bg = \pi(\gamma).
\]
For each isomorphism type of $F$ we give the following data:
\begin{enumerate}
\item[(i)] The presentation of $F$.
\item[(ii)] The family number with the given holonomy group according to \cite{D96}. (If the family is modelled on a two-step nilpotent group, only a number is given. For three-step nilpotent groups, the number is prefixed by the letter `B'. This agrees with the notation given by the GAP package Aclib \cite{aclib}.)
\item[(iii)] The character of the holonomy representation. (The character tables are given in the appendix.)
\item[(iv)] The  integral holonomy representation given by the images of the generators either in the form $\theta \colon F \to \SO(n,\Z)$, if such a representation exists, or the form $\theta \colon F \to \SL(n,\Z)$, otherwise.
\item[(v)] The structure of the group $\tf = \lambda^{-1}(F)$ and the homomorphism $\lambda_{|\tf}$. (We will make use of the central extension 
\[
\sesr{1}{\{ \pm 1 \}}{\tf}{\lambda_{|\tf}}{F}{1}
\]
in finding a presentation of $\tf$ in $\Spin(4)$.)
\end{enumerate}

\subsection{Preimages of the holonomy and their representations in $\bf \Spin(4)$.}

\begin{enumerate}
\setlength\itemsep{.5\baselineskip}
\setlength{\parindent}{0pt}
\item $F = 1$. 

Families: 1, B1. 

The character: $4\chi_1$.

$\tf = \{ \pm 1 \} \cong C_2$, $\lambda(\pm 1) = \theta(1)$.

\item $F = \langle \ba \;|\; \ba^2=1 \rangle \cong C_2$.

Families: 3, 4, 5, 7b, 9b, B3, B3b, B3c, B4, B5, B5b. 

In each family, the character is equal to 
\[
2\chi_1 + 2\chi_2.
\]

If we take
\[
\theta(\ba) = \diag(1,1,-1,-1),
\]
then $\tf = \langle e_3e_4 \rangle \cong C_4$ and $\lambda(e_3e_4) = \theta(\ba)$.

\item $F = \langle \ba, \bb \;|\; \ba^2=\bb^2=1, \ba\bb=\bb\ba \rangle \cong C_2 \times C_2$. 

Families: 27, 29b, 30, 32, 33b, 34, 37, 41, 43, 45. 

In each family, the character is equal to 
\[
\chi_1 + \chi_2 + \chi_3 + \chi_4.
\]
If we take
\[
\theta(\ba) = \diag(-1,-1,1,1),\; \theta(\bb) = \diag(1,-1,-1,1),
\]
then
\[
\lambda(\pm e_1e_2) = \theta(\ba),\; \lambda(\pm e_2e_3) = \theta(\bb)
\]
and $\tf \cong Q_8$. Identifying $Q_8 = \{ \pm 1, \pm i, \pm j, \pm k \}$ with $\tf$, 
we have
\[
\lambda(i) = \theta(\ba) \text{ and } \lambda(j) = \theta(\bb).
\]

\item $F = \langle \ba \;|\; \ba^4=1 \rangle \cong C_4$.

Families: 75, 76, 77, 79, 80.

In each family, the character is equal to:
\[
2\chi_1+\chi_3+\chi_4.
\]
Take
\[
\theta(\ba) = 
\begin{bmatrix}
0 & -1 & 0 & 0\\
1 &  0 & 0 & 0\\
0 &  0 & 1 & 0\\
0 &  0 & 0 & 1\\
\end{bmatrix}
\]
and
\[
c =  \frac{1+e_1e_2}{\sqrt{2}}.
\]
We obtain $\tf = \langle c \;|\; c^8=1 \rangle \cong C_8$ and $\lambda(c) = \theta(\ba)$.

\item $F = \langle \ba, \bb \;|\; \ba^4=\bb^2=(\ba\bb)^2=1 \rangle \cong D_8$.

Families: 103, 104, 106, 110.

In each family, the character is equal to
\[
\chi_1+\chi_2+\chi_5.
\]
If we take
\[
\theta(\ba) = 
\begin{bmatrix}
1 & 0 & 0 & 0\\
0 & 0 &-1 & 0\\
0 & 1 & 0 & 0\\
0 & 0 & 0 & 1
\end{bmatrix}, \;
\theta(\bb) = \diag(-1,1,-1,1),
\]
then
\[
\lambda\left( \pm \frac{1+e_2e_3}{\sqrt{2}} \right) = \theta(\ba) \text{ and } \lambda(\pm e_1e_3) = \theta(\bb).
\]
Taking $a = (1+e_2e_3)/{\sqrt{2}}, b=e_1e_3, c=-1$, we obtain
\[
\tf = \langle a,b,c \;|\; c^2=[a,c]=[b,c]=1, a^4=b^2=(ab)^2=c \rangle \cong Q_{16}.
\]

\item $F = \langle \ba \;|\; \ba^3=1 \rangle \cong C_3$.

Families: 144, 146.

$F$ is an odd-order group, so we know that there exists $c \in \Spin(4)$ such that
\[
\tf = \langle c | c^6 = 1 \rangle \text{ and } \lambda(c) = \theta(\ba).
\]

\item $F = \langle \ba, \bb \;| \; \ba^3=\bb^2=(\bb\ba)^2 = 1 \rangle \cong S_3$.

Families: 158, 159, 161.

In each family, the character is equal to 
\[
\chi_1+\chi_2+\chi_3.
\]
Let us consider for example the integral holonomy representation for the family no.~161:
\[
\theta(\ba) = 
\begin{bmatrix}
1 & 0 & 0 & 0\\
0 & 0 & 1 & 0\\
0 & 1 & 0 & 0\\
0 & 0 & 0 & 1
\end{bmatrix}, \;
\theta(\bb) = 
\begin{bmatrix}
-1 & 0 & 0 & 0\\
0 & 0 & 1 & 0\\
0 & 1 & 0 & 0\\
0 & 0 & 0 & 1
\end{bmatrix}.
\]
Taking $c=-1$, 
\[
a = \frac{(1+e_2e_4)(1+e_2e_3)e_3e_4}{2} = \frac{(e_4-e_2)(e_2-e_3)}{2}
\]
and
\[
b = \frac{e_1e_2(1+e_2e_3)}{2} = \frac{e_1(e_2-e_3)}{2}
\]
we get
\[
\lambda\left( \pm a \right) = \theta(\ba) \text{ and } \lambda\left( \pm b \right) = \theta(\bb).
\]
Moreover
\[
\tf = \langle a,b,c \;|\; c^2=[a,c]=[b,c]=1, a^3 = b^2 = (ba)^2 = c \rangle \cong C_3 \rtimes C_4
\]
where $C_3 = \langle a^2 \rangle$ and $C_4 = \langle b \rangle$.

\item $F = \langle \ba \;|\; \ba^6 = 1 \rangle \cong C_6$.

Families: 168, 169, 172, 173.

In each family, the character is equal to 
\[
2\chi_1 + \chi_5+\chi_6.
\]
There is no representation of $C_6$ to $\SO(4,\Z)$ with the given character.
Consider the integral holonomy representation of the family no. 168:
\[
\theta(\ba) = 
\begin{bmatrix}
1 & 0 & 0 & 0\\
0 & 0 & 1 & 0\\
0 &-1 & 1 & 0\\
0 & 0 & 0 & 1
\end{bmatrix}.
\]
Out of the 5 groups of order 12, only  $C_{12}$ and $C_2 \times C_6$ are central extensions of $C_2$ and $C_6$, but
\[
\theta(\ba^3) = \diag(1,-1,-1,1)
\]
and hence $\tf = \langle a \;|\; a^{12} = 1 \rangle$ and
\[
\lambda(a) = \theta(\ba).
\]

\item $F = \langle \ba, \bb \;|\; \ba^6 = \bb^2 = (\bb\ba)^2 = 1 \rangle \cong D_{12}$.

Family: 184.

The character is equal to 
\[
\chi_1+\chi_2+\chi_6.
\]

The representation looks as follows:
\[
\theta(\ba) = 
\begin{bmatrix}
1 & 0 & 0 & 0\\
0 & 0 & 1 & 0\\
0 &-1 & 1 & 0\\
0 & 0 & 0 & 1
\end{bmatrix}, \;
\theta(\bb) = 
\begin{bmatrix}
-1 & 0 & 0 & 0\\
0 & 0 &-1 & 0\\
0 &-1 & 0 & 0\\
0 & 0 & 0 & 1
\end{bmatrix}.
\]
Since $\theta(\bb\ba)$ has exactly two eigenvalues equal to $-1$, its preimage in $\Spin(4)$ is an element of order $4$. We obtain the following presentation of $\tf$:
\[
\tf = \langle a,b,c \;|\; c^2 = [a,c] = [b,c] = 1, a^6 = b^2 = (ba)^2 = c \rangle \cong C_3 \rtimes Q_8,
\]
where $C_3 = \langle a^2c \rangle$ and $Q_8 = \langle a^3c, b \rangle$.
\end{enumerate}

\section{Results}

In Table 1, we give the complete information on spin structures of four-dimensional almost flat manifolds. Suppose $M$ is an almost-flat manifold of dimension four with the fundamental group $\G$ that fits into the short exact sequence \eqref{eq:abgroup}. In the table, the family of the group is indicated by  `Fam' column. The associated parameters are given in `Params' column but only mod 2, as this makes no difference in the classification (see Proposition \ref{mod2}). Note that if $\G$ is modelled on a 3-step nilpotent group, the first parameter describes the almost crystallographic group $Q=\G/{\sqrt[\Lambda]{\gamma_3(\Lambda)}}$ (see chapters 6.3 and 7.3 of \cite{D96}).
The number of spin structures is given in the column `\# S'. The families that do not admit a spin structure are highlighted in bold.
In addition, we give an isomorphism type of the holonomy group of $\G$ in the column `Hol'.
\begin{cor} \label{trivial} All spin families (i.e.~not in bold)  in Table 1 are parallelizable.
\end{cor}
\begin{proof} By a result of Hirzebruch and Hopf \cite{HH58}, a closed orientable $4$-manifold $M$ is parallelizable if and only if its Euler characteristic, signature (or equivalently, first Pontryagin class) and the second Stiefel-Whitney class vanish.  If $M$ is also infra-nil, then it is finitely covered by a nilmanifold which is parallelizable. It follows that $M$  must have vanishing Euler characteristic and signature. Therefore, $M$ is parallelizable if and only if it has a spin structure. 
\end{proof}

\begin{table}[h]
\footnotesize
\newcommand{\p}{\phantom{a}}
\def\arraystretch{1.2}
\begin{multicols}{3}
\begin{tabular}{RLLE}
Fam & Hol & Params & \# S \\ \hline
1\p & $1$ & $( 0, 0, 0 )$ & 16 \\
1\p & $1$ & $( 1, 0, 0 )$ & 8 \\
3\p & $C_2$ & $( 0, 0, 0, 1 )$ & 16 \\
4\p & $C_2$ & $( 0, 0, 0, 0 )$ & 16 \\
4\p & $C_2$ & $( 0, 1, 0, 0 )$ & 8 \\
4\p & $C_2$ & $( 1, 0, 0, 0 )$ & 8 \\
5\p & $C_2$ & $( 0, 0, 0, 1 )$ & 8\emrow \\ 
5\p & $C_2$ & $( 1, 0, 0, 1 )$ & 0\eemrow \\
7b & $C_2$ & $( 0, 0, 0, 0 )$ & 16 \\
7b & $C_2$ & $( 0, 0, 1, 0 )$ & 8 \\
7b & $C_2$ & $( 0, 1, 0, 0 )$ & 8 \\
7b & $C_2$ & $( 0, 1, 1, 0 )$ & 8 \\
7b & $C_2$ & $( 1, 0, 0, 0 )$ & 8 \\
9b & $C_2$ & $( 0, 0, 0, 0 )$ & 8 \\
9b & $C_2$ & $( 0, 1, 0, 0 )$ & 4 \\
9b & $C_2$ & $( 1, 0, 0, 0 )$ & 4 \\
27\p & $C_2^2$ & $( 0, 0, 0, 1, 0 )$ & 16 \\
29b & $C_2^2$ & $( 0, 0, 0, 0, 0 )$ & 16 \\
29b & $C_2^2$ & $( 0, 0, 1, 0, 0 )$ & 8 \\
29b & $C_2^2$ & $( 0, 1, 0, 0, 0 )$ & 8 \\
29b & $C_2^2$ & $( 1, 0, 0, 0, 0 )$ & 8 \\
29b & $C_2^2$ & $( 1, 0, 1, 0, 0 )$ & 8 \\
29b & $C_2^2$ & $( 1, 1, 0, 0, 0 )$ & 8 \\
30\p & $C_2^2$ & $( 0, 0, 1, 0, 0 )$ & 8 \emrow\\
30\p & $C_2^2$ & $( 1, 0, 1, 0, 0 )$ & 0\eemrow\\
32\p & $C_2^2$ & $( 0, 1, 0, 0, 0 )$ & 8 \emrow\\
32\p & $C_2^2$ & $( 0, 1, 1, 0, 0 )$ & 0\eemrow\\
33b & $C_2^2$ & $( 0, 0, 0, 0, 0 )$ & 8 \\
33b & $C_2^2$ & $( 0, 0, 1, 0, 0 )$ & 4 \\
33b & $C_2^2$ & $( 0, 1, 0, 0, 0 )$ & 8 \\
33b & $C_2^2$ & $( 1, 0, 0, 0, 0 )$ & 4 \\
33b & $C_2^2$ & $( 1, 0, 1, 0, 0 )$ & 4 \\
34\p & $C_2^2$ & $( 0, 0, 1, 0, 0 )$ & 8\emrow \\
34\p & $C_2^2$ & $( 1, 0, 1, 0, 0 )$ & 0\eemrow\\
37\p & $C_2^2$ & $( 0, 0, 1, 0, 0 )$ & 8\emrow \\
41\p & $C_2^2$ & $( 0, 0, 1, 0, 0 )$ & 0\\
41\p & $C_2^2$ & $( 0, 1, 1, 0, 0 )$ & 0\eemrow\\
41\p & $C_2^2$ & $( 1, 0, 1, 0, 0 )$ & 8\emrow \\
41\p & $C_2^2$ & $( 1, 1, 1, 0, 0 )$ & 0\eemrow\\
43\p & $C_2^2$ & $( 0, 0, 1, 0, 0 )$ & 4\emrow \\
43\p & $C_2^2$ & $( 1, 0, 1, 0, 0 )$ & 0 \\
45\p & $C_2^2$ & $( 0, 0, 1, 0, 0 )$ & 0\eemrow \\
45\p & $C_2^2$ & $( 0, 1, 1, 0, 0 )$ & 8 \\
\end{tabular}

\begin{tabular}{RLLE}
Fam & Hol & Params & \# S \\ \hline
75 & $C_4$ & $( 0, 0, 0, 1 )$ & 8 \\
76 & $C_4$ & $( 0, 0, 0, 0 )$ & 8 \\
76 & $C_4$ & $( 0, 1, 0, 0 )$ & 4 \\
76 & $C_4$ & $( 1, 0, 0, 0 )$ & 4 \\
77 & $C_4$ & $( 0, 0, 0, 1 )$ & 8 \\
79 & $C_4$ & $( 0, 0, 0, 1 )$ & 4 \\
80 & $C_4$ & $( 0, 0, 0, 1 )$ & 4\emrow \\
80 & $C_4$ & $( 1, 0, 0, 1 )$ & 0\eemrow \\
103 & $D_8$ & $( 0, 0, 0, 1, 0 )$ & 8 \\
104 & $D_8$ & $( 0, 0, 1, 0, 0 )$ & 4 \\
106 & $D_8$ & $( 0, 0, 1, 0, 0 )$ & 4 \\
106 & $D_8$ & $( 0, 0, 1, 1, 0 )$ & 4\emrow \\
110 & $D_8$ & $( 0, 0, 1, 0, 0 )$ & 0\eemrow \\
110 & $D_8$ & $( 1, 0, 1, 0, 0 )$ & 4 \\
143 & $C_3$ & $( 0, 0, 0, 0 )$ & 4 \\
143 & $C_3$ & $( 0, 0, 0, 1 )$ & 4 \\
143 & $C_3$ & $( 1, 0, 0, 0 )$ & 2 \\
143 & $C_3$ & $( 1, 0, 0, 1 )$ & 2 \\
144 & $C_3$ & $( 0, 0, 0, 0 )$ & 4 \\
144 & $C_3$ & $( 0, 1, 0, 0 )$ & 4 \\
144 & $C_3$ & $( 1, 0, 0, 0 )$ & 2 \\
144 & $C_3$ & $( 1, 1, 0, 0 )$ & 2 \\
146 & $C_3$ & $( 0, 0, 0, 0 )$ & 4 \\
146 & $C_3$ & $( 0, 0, 0, 1 )$ & 4 \\
146 & $C_3$ & $( 1, 0, 0, 0 )$ & 2 \\
146 & $C_3$ & $( 1, 0, 0, 1 )$ & 2 \\
158 & $S_3$ & $( 0, 0, 0, 0, 0 )$ & 4 \\
158 & $S_3$ & $( 0, 0, 0, 1, 0 )$ & 4 \\
158 & $S_3$ & $( 0, 1, 0, 1, 0 )$ & 4 \\
158 & $S_3$ & $( 1, 0, 0, 0, 0 )$ & 2 \\
158 & $S_3$ & $( 1, 0, 0, 1, 0 )$ & 2 \\
158 & $S_3$ & $( 1, 1, 0, 1, 0 )$ & 2 \\
159 & $S_3$ & $( 0, 0, 0, 0, 0 )$ & 4 \\
159 & $S_3$ & $( 0, 0, 1, 0, 0 )$ & 4 \\
159 & $S_3$ & $( 1, 0, 0, 0, 0 )$ & 2 \\
159 & $S_3$ & $( 1, 0, 1, 0, 0 )$ & 2 \\
161 & $S_3$ & $( 0, 0, 0, 0, 0 )$ & 4 \\
161 & $S_3$ & $( 0, 1, 0, 0, 0 )$ & 4 \\
161 & $S_3$ & $( 1, 0, 0, 0, 0 )$ & 2 \\
161 & $S_3$ & $( 1, 1, 0, 0, 0 )$ & 2 \\
168 & $C_6$ & $( 0, 0, 0, 1 )$ & 4 \\
169 & $C_6$ & $( 0, 0, 0, 0 )$ & 4 \\
\end{tabular}

\begin{tabular}{RLLE}
Fam & Hol & Params & \# S \\ \hline
169\p & $C_6$ & $( 1, 0, 0, 0 )$ & 2 \\
172\p & $C_6$ & $( 0, 0, 0, 1 )$ & 4 \\
173\p & $C_6$ & $( 0, 0, 0, 0 )$ & 4 \\
173\p & $C_6$ & $( 0, 0, 0, 1 )$ & 4 \\
173\p & $C_6$ & $( 1, 0, 0, 0 )$ & 2 \\
173\p & $C_6$ & $( 1, 0, 0, 1 )$ & 2 \\
184\p & $D_{12}$ & $( 0, 0, 0, 1, 0 )$ & 4 \\
B1\p & $1$ & $( 0, 0, 0, 0 )$ & 16 \\
B1\p & $1$ & $( 0, 0, 0, 1 )$ & 8 \\
B1\p & $1$ & $( 0, 0, 1, 0 )$ & 8 \\
B1\p & $1$ & $( 0, 0, 1, 1 )$ & 8 \\
B1\p & $1$ & $( 0, 1, 0, 0 )$ & 8 \\
B1\p & $1$ & $( 0, 1, 0, 1 )$ & 8 \\
B1\p & $1$ & $( 0, 1, 1, 0 )$ & 8 \\
B1\p & $1$ & $( 0, 1, 1, 1 )$ & 8 \\
B1\p & $1$ & $( 1, 0, 0, 0 )$ & 8 \\
B1\p & $1$ & $( 1, 0, 0, 1 )$ & 4 \\
B1\p & $1$ & $( 1, 0, 1, 0 )$ & 4 \\
B1\p & $1$ & $( 1, 0, 1, 1 )$ & 4 \\
B1\p & $1$ & $( 1, 1, 0, 0 )$ & 8 \\
B1\p & $1$ & $( 1, 1, 0, 1 )$ & 4 \\
B1\p & $1$ & $( 1, 1, 1, 0 )$ & 4 \\
B1\p & $1$ & $( 1, 1, 1, 1 )$ & 4 \\
B3\p & $C_2$ & $( 1, 0, 0, 0, 1 )$ & 8 \\
B3b & $C_2$ & $( 1, 0, 0, 0, 1 )$ & 8\emrow \\
B3b & $C_2$ & $( 1, 1, 0, 0, 1 )$ & 0\eemrow \\
B3c & $C_2$ & $( 0, 0, 0, 0, 1 )$ & 16 \\
B3c & $C_2$ & $( 1, 0, 0, 0, 1 )$ & 16 \\
B4\p & $C_2$ & $( 0, 0, 0, 0, 0 )$ & 16 \\
B4\p & $C_2$ & $( 0, 0, 0, 1, 0 )$ & 8 \\
B4\p & $C_2$ & $( 0, 0, 1, 0, 0 )$ & 8 \\
B4\p & $C_2$ & $( 0, 1, 0, 0, 0 )$ & 8 \\
B4\p & $C_2$ & $( 1, 0, 0, 0, 0 )$ & 8 \\
B4\p & $C_2$ & $( 1, 0, 0, 1, 0 )$ & 4 \\
B4\p & $C_2$ & $( 1, 0, 1, 0, 0 )$ & 8 \\
B4\p & $C_2$ & $( 1, 1, 0, 0, 0 )$ & 4 \\
B5\p & $C_2$ & $( 0, 0, 0, 0, 1 )$ & 8\emrow\\
B5\p & $C_2$ & $( 0, 1, 0, 0, 1 )$ & 0\eemrow \\
B5\p & $C_2$ & $( 1, 0, 0, 0, 1 )$ & 8\emrow \\
B5\p & $C_2$ & $( 1, 1, 0, 0, 1 )$ & 0\eemrow \\
B5b & $C_2$ & $( 1, 0, 0, 0, 1 )$ & 4\emrow \\
B5b & $C_2$ & $( 1, 1, 0, 0, 1 )$ & 0\eemrow \\
\end{tabular}

\end{multicols}
\caption{Spin structures on 4-dimensional almost flat manifolds}
\end{table}

\restoregeometry

\appendix 

\section{Character tables of holonomy groups of 4-dimensional almost Bieberbach groups}
\columnseprule=.4pt
\begin{multicols}{2}
\newcommand{\bxi}{\bar{\xi}}
\begin{enumerate}[label=\arabic{enumi}.]

\item Character of the trivial group is denoted by $\chi_1$.

\item $C_2 = \langle \ba \st \ba^2 = 1 \rangle$ 
\[
\begin{array}{l|rr}
       & 1 & \ba \\ \hline
\chi_1 & 1 & 1 \\
\chi_2 & 1 &-1
\end{array}
\]

\item $C_2^2 = \langle \ba,\bb \st \ba^2=\bb^2=[\ba,\bb]=1 \rangle$
\[
\begin{array}{l|rrrr}
       &   1 & \ba & \bb & \ba\bb \\ \hline
\chi_1 &   1 &   1 &   1 &  1 \\
\chi_2 &   1 &  -1 &   1 & -1 \\
\chi_3 &   1 &   1 &  -1 & -1 \\
\chi_4 &   1 &  -1 &  -1 &  1 \\
\end{array}
\]

\item $C_4 = \langle \ba \st \ba^4 = 1 \rangle$
\[
\begin{array}{l|rrrr}
       & 1 & \ba & \ba^2 & \ba^3 \\ \hline
\chi_1 & 1 &   1 &    1 &     1 \\
\chi_2 & 1 &  -1 &    1 &    -1 \\
\chi_3 & 1 &   i &   -1 &    -i \\
\chi_4 & 1 &  -i &   -1 &     i \\
\end{array}
\]

\item $D_8 = \langle \ba, \bb | \ba^4=\bb^2=(\ba\bb)^2=1 \rangle$
\[
\begin{array}{l|rrrrr}
       & 1 & \bb & \ba\bb & \ba^2 & \ba \\ \hline
\chi_1 & 1 &   1 &      1 &     1 &   1 \\
\chi_2 & 1 &  -1 &      1 &     1 &  -1 \\
\chi_3 & 1 &   1 &     -1 &     1 &  -1 \\
\chi_4 & 1 &  -1 &     -1 &     1 &   1 \\
\chi_5 & 2 &   0 &      0 &    -2 &   0 \\
\end{array}
\]

\item $C_3 = \langle \ba \st \ba^3 = 1 \rangle$
\[
\begin{array}{l|rrr}
       & 1 &  \ba & \ba^2\\ \hline
\chi_1 & 1 &    1 &     1 \\
\chi_2 & 1 &  \xi &  \bxi \\
\chi_3 & 1 & \bxi &   \xi \\
\end{array}
\]
where $\xi=e^{2\pi i/3}$.

\item $S_3 = \langle \ba, \bb \st \ba^2 = \bb^3 = (\ba\bb)^2 = 1 \rangle$
\[
\begin{array}{l|rrr}
       & 1 &  \ba & \bb\\ \hline
\chi_1 & 1 &    1 &   1 \\
\chi_2 & 1 &   -1 &   1 \\
\chi_3 & 2 &    0 &  -1 \\
\end{array}
\]

\item{} $C_6 = \langle \ba \st \ba^6=1 \rangle$
\[
\begin{array}{l|rrrrrr}
       & 1 & \ba & \ba^2 & \ba^3 & \ba^4 & \ba^5\\ \hline
\chi_1 & 1 &   1 &     1 &     1 &     1 &     1\\
\chi_2 & 1 &  -1 &     1 &    -1 &     1 &    -1\\
\chi_3 & 1 &  -1 &   \xi &  -\xi &  \bxi & -\bxi\\
\chi_4 & 1 &  -1 &  \bxi & -\bxi &   \xi &  -\xi\\
\chi_5 & 1 &   1 &   \xi &   \xi &  \bxi &  \bxi\\
\chi_6 & 1 &   1 &  \bxi &  \bxi &   \xi &   \xi\\
\end{array}
\]
where $\xi=e^{4\pi i/3}$.

\item $D_{12} = \langle \ba, \bb \st \ba^2 = \bb^6 = (\ba\bb)^2 = 1 \rangle$
\[
\begin{array}{l|rrrrrr}
       & 1 & \ba & \bb^3 & \bb^2 & \ba\bb & \bb\\ \hline
\chi_1 & 1 &   1 &     1 &     1 &      1 &   1\\
\chi_2 & 1 &  -1 &    -1 &     1 &      1 &  -1\\
\chi_3 & 1 &  -1 &     1 &     1 &     -1 &   1\\
\chi_4 & 1 &   1 &    -1 &     1 &     -1 &  -1\\
\chi_5 & 2 &   0 &    -2 &    -1 &      0 &   1\\
\chi_6 & 2 &   0 &     2 &    -1 &      0 &  -1\\
\end{array}
\]

\end{enumerate}

\end{multicols}

\section*{Acknowledgements}

The computations were performed with usage of GAP \cite{GAP} and the GAP package Aclib \cite{aclib}.

\end{document}